\documentclass[12pt]{paper}

\usepackage{amssymb,amsfonts,amsthm,amsmath,mathrsfs}

\makeatletter
\newcommand{\subjclass}[2][1991]{%
  \let\@oldtitle\@title%
  \gdef\@title{\@oldtitle\footnotetext{#1 \emph{Mathematics subject classification.} #2}}%
}
\makeatother

\usepackage{graphicx, verbatim}
\usepackage{xcolor}

\usepackage[a4paper, hmargin=3cm, vmargin={4cm, 4cm}]{geometry}
\usepackage{fancyhdr}

\newtheorem{theorem}{Theorem}
\newtheorem{lemma}[theorem]{Lemma}
\newtheorem{question}[theorem]{Question}

\newtheorem{claim}[theorem]{Claim}
\newtheorem{remark}[theorem]{Remark}

\newcommand{\bR}{\mathbb R}
\newcommand{\bC}{\mathbb C}
\newcommand{\bZ}{\mathbb Z}
\newcommand{\bD}{\mathbb D}

\newcommand{\bH}{\mathbb H}
\newcommand{\bG}{\mathbb G}

\newcommand{\la}{\lambda}
\renewcommand{\phi}{\varphi}

\newcommand{\im}{{\rm Im}\,}
\newcommand{\re}{{\rm Re}\,}

\begin{document}

\title{
When a meromorphic function that omits three values is of bounded type
}

\author{Alexandre Eremenko, Aleksei Kulikov, Mikhail Sodin}

\maketitle

\abstract{Suppose that a function $F$ is meromorphic in the domain $\bH(-m)=\{z\colon \im z>-m(\re z) \}$, 
where $m$ is an even, positive, and continuous function that does not increase on
$\bR_{\ge 0}$,  and suppose that $F$ omits there three distinct values. 
Then $F$ is of bounded type in the upper half-plane
(i.e., is represented there as a quotient of two bounded analytic functions), provided that the logarithmic integral 
of the function $m$ is convergent. On the other hand, if the logarithmic integral of $m$ diverges, there exists a
function $F$ meromorphic in $\bH(-m)$, that omits there three distinct values,
and which is of unbounded type in the upper half-plane.

This result is motivated by a century-old question originating with Rolf Nevanlinna.
}

\section{Introduction}

We use the following notation:

\begin{itemize}
\item  $m\colon \bR \to \bR_{> 0}$ is an even continuous  function that does not increase on
$\bR_{\ge 0}$.

\item $\bH$ denotes the upper half-plane;

\item $\bH(\pm m) = \{z=x+iy\colon y > \pm m(x), x\in \bR\}$.

\item For $x\in\bR$, $x^+ = \max\{x, 0\}$ and $x^- = (-x)^+$.

\item $A \lesssim B$ means that $A\le CB$ with a positive constant $C$, $A \gtrsim B$
means that $B \lesssim A$, and $A \simeq B$ means that $A \lesssim B$
and $A \gtrsim  B$ hold simultaneously.
\end{itemize}

A meromorphic function $F$ is said to be of {\em bounded type} (belongs to the {\em Nevanlinna class}) 
in a domain $G\subset \bC$  if it can be represented as $F=F_1/F_2$, where $F_1, F_2$ are bounded analytic 
functions in $G$. Otherwise, $F$ is said to be of {\em unbounded type} in $G$. In this note we prove the 
following theorem.

\begin{theorem}\label{main_thm} \mbox{}

\medskip\noindent {\rm (a)} Assume that $\displaystyle \int_1^\infty \frac{\log^- m(t)}{t^2}\, {\rm d}t<\infty$.
Then any meromorphic function $F$ in $\bH(-m)$ that omits three distinct values  in $\bH(-m)$
is of bounded type in $\bH$.

\smallskip\noindent {\rm (b)} Assume that $\displaystyle \int_1^\infty \frac{\log^- m(t)}{t^2}\, {\rm d}t =\infty$. 
Then there exists a function $F$ meromorphic in $\bH(-m)$ that omits three distinct values in $\bH(-m)$ and is of 
unbounded type in $\bH$.

\end{theorem}
This result arose as a by-product of our attempt to (negatively) resolve the following question originating in Nevanlinna's work~\cite{Nev} (cf.~\cite[item 6]{Ostr}):

\begin{question}\label{quest_Nevanlinna}
Let $F$ be a meromorphic function in $\bC$.
Suppose that it omits in a half-plane three distinct values from the extended complex plane.
Does it follow that $F$ is of bounded type in that half-plane?
\end{question}

It is well-known that the elliptic modular function~\cite[Chapter~7, Section~3.4]{Ahlfors} is analytic in the upper half-plane, 
omits there the values $0$ and $1$, and is not of bounded type there (we will provide details while proving part (b) 
of Theorem~\ref{main_thm}). It is also 
well-known~\cite[Chapter~VI \S4 and Chapter~VII \S1]{Nev-book} that if a meromorphic function
in a simply-connected planar domain  omits a set of positive logarithmic capacity, then it is of bounded type. 

It is also worth mentioning that it 
was a surprise for us to see appearance of the logarithmic integral in this setting\footnote
{In that occassion, let us recall here what Paul Koosis wrote in the preface to~\cite{Koosis}:
{\em
I have loved $\displaystyle \int_{-\infty}^\infty (\log M(t)/(1+t^2) {\rm d}t$ - the logarithmic integral - 
ever since I first read Szeg\H{o}'s discussion about the geometric mean of a function and the theorem 
named after him in his book on orthogonal polynomials, over 30 years ago. Far from being an isolated artifact, this object playes an important role in many diverse and seemingly unrelated investigations 
about functions of one real or complex variable, and a serious account of its appearances would involve
a good deal of the analysis done since 1900.
}}.

\begin{remark}
\rm After this paper was submitted, a preprint \cite{HZ} appeared 
containing a negative answer to Question~\ref{quest_Nevanlinna}.
In that work, Yixin He and Teng Zhang  constructed a meromorphic function $F$ 
in the plane which omits $0,1$ and $\infty$ in $\bH$ and is not of bounded type in $\bH$.
We will briefly discuss their construction in Section~\ref{sect:brief}.
\end{remark}

\paragraph*{Acknowledgement.}\ 
We thank the referee, whose questions and comments helped us improve the presentation.
We thank Yixin He and Teng Zhang, who informed us about their work~\cite{HZ} and
drew our attention to an inaccuracy in the statement of Edrei's result in the original version of this work.

The work of Aleksei Kulikov was supported by the United States – Israel Binational Science Foundation 
BSF Grant 2020019, the Israel Science Foundation ISF Grant 1288/21, and by the VILLUM Centre of Excellence 
for the Mathematics of Quantum Theory (QMATH) with Grant No.10059.  
The work of Mikhail Sodin was supported by the Israel Science Foundation ISF grants 1288/21, 2319/25,
and the United States – Israel Binational Science Foundation BSF grants 2020019, 2024142.

\section{Functions of bounded type in the upper half-plane}

\subsection{Nevanlinna's work}
The question goes back to Nevanlinna's work~\cite{Nev} 
on the distribution of values of meromorphic functions in a half-plane. Therein, Nevanlinna introduced the 
characteristics of functions meromorphic in the {\em closed} upper half-plane $\overline\bH$. 
Let $f$ be a function meromorphic in $\overline\bH$ (that is, meromorphic in a domain 
$G_f \supset \overline\bH$), and $\{w_n\}$ be the set of its poles in $\bH$, 
counted with multiplicities. 
Then
\begin{align*}
c(r; f) &= \sum_{1<|w_n|\le r} \sin(\arg w_n), \\
C(r; f) &= 2\int_1^r c(t; f) \left( \frac1{t^2} + \frac1{r^2}\right)\, {\rm d}t \\
&= 2\sum_{1<|w_n|<r} \left( \frac1{|w_n|} - \frac{|w_n|}{r^2} \right) \sin(\arg w_n), \\
A(r; f) &= \frac1{\pi}\, \int_1^r \left( \frac1{t^2} - \frac1{r^2} \right) \bigl[ \log^+ |f(t)| + \log^+|f(-t)| \bigr]\,
{\rm d}t,  \\
B(r; f) &= \frac2{\pi r}\, \int_0^\pi \log^+|f(re^{i\theta})| \sin \theta\, {\rm d}\theta, \\
S(r; f) &= A(r; f) + B(r; f) + C(r; f).
\end{align*}
We list the basic asymptotic properties of these characteristics for $r\to\infty$:
\begin{enumerate}
\item If $D$ is any of the characteristics $A$, $B$, and $C$, then $D(r; f_1 f_2) \le D(r; f_1) + D(r; f_2)$
and $D(r; f_1 + f_2) \le D(r; f_1) + D(r; f_2) + \log 2$.
In particular, $S(r; f_1 f_2) \le S(r; f_1) + S(r; f_2)$ and $S(r; f_1+f_2)\le S(r; f_1) + S(r; f_2) + 3\log 2$.
\item For $a\in \bC$, 
$ S(r; 1/(f-a)) = S(r; f) + O(1)$ with an implicit constant depending on~$a$.
\item There exists a non-decreasing function $S_o(r; f)$ such that $S(r; f) = S_o(r; f) + O(1)$.
\item $S(r; f)$ is bounded if and only if the function $f$ is of bounded type in $\bH$.
\end{enumerate}
The first property is straightforward. The second one is an analogue of the first fundamental theorem.
The proof of the second property~\cite[Chapter~I, Theorem~5.1]{GO} is based on the Carleman integral formula.
One way to prove the third property is to introduce the Ahlfors--Shimizu geometric form of the characteristics
$S$:
\begin{equation*}\label{eq:AS}
S_o(r; f) = \frac1\pi\, \int_1^r \left( \frac1t - \frac{t}{r^2} \right)  \left[ \,
\int_0^{\pi} \rho_f(te^{i\theta})^2 \,
\sin\theta\, {\rm d}\theta  \, \right]  t\, {\rm d}t 
\end{equation*}
where $\rho_f = |f'|/(1+|f|^2)$ is the spherical derivative of $f$. The function $r\mapsto S_o(r; f)$ is
non-decreasing, and a computation~\cite[Chapter~I, Theorem~5.4]{GO} shows that 
\[
S(r; f) = S_o(r; f) + O(1), \quad r\to\infty.
\]
In the fourth property, the direction 
\[
f \ {\rm is\ of\ a\ bounded\ type\ } \Longrightarrow S(r; f) \ {\rm is\ bounded}  
\]
follows from the Nevanlinna factorization of the functions of bounded type in $\bH$, since for
each factor in the Nevanlinna factorization the characteristics $S$ stays bounded. The opposite 
direction was proven in~\cite[Satz 1, p. 31]{Nev}\footnote{
Here is a brief sketch of the proof. Boundedness of the characteristics $C(r; f)$ implies that the poles of  
$f$ satisfy the Blaschke condition in $\bH$. By property 2 of $S(r; f)$, it also implies boundedness of 
$C(r; 1/f)$, i.e., the zeroes of $f$ also satisfy the Blaschke condition. Therefore, we can separate from 
$f$ the Blaschke products $\mathscr B_0$ and $\mathscr B_\infty$ corresponding to its zeroes and poles.  Similarly,
boundedness of $A(r, f)$ and $A(r, 1/f)$ allows us to define the outer factor $\mathscr O_f$ such that 
$|\mathscr O_f(x)| = |f(x)|$, $x\in\bR$. Separating these factors from $f$,
we arrive at a meromorphic function $G$ in $\overline\bH$, having neither zeroes, nor poles
in $\bH$ and satisfying $|G|=1$ on $\bR$. Let $h=\log |G|$. It is    
a harmonic function in $\bH$ that continuously vanishes on $\bR$
and satisfies 
\[
\int_0^\pi h^+(re^{i\theta}) \sin\theta \, {\rm d}\theta = O(r), \quad r\to\infty. 
\]
By the Phragm\'en--Lindel\"of principle, 
this yields that $h(z)=a\, {\rm Im\,} z$ with $a\in \bR$. Thus, $G(z)=\Theta e^{iaz}$, where $\Theta$ is a
unimodular constant. Then
\[
f(z) = \Theta \frac{\mathscr B_0(z)}{\mathscr B_\infty (z)}\, \mathscr O_f(z) e^{iaz},
\] 
and therefore, $f$ is of bounded type.
}.

\medskip
Nevanlinna showed~\cite[Satz~II and Hilfssatz, p. 18]{Nev} 
that {\em if $F$ is a meromorphic function in $\bC$ 
of finite order of growth which omits in $\bH$ three distinct values, then}
$S(r; F) = O(1)$ for $r\to\infty$, and therefore, $F$ is of bounded type in $\bH$. 
That is, Question~2 has a positive answer provided that $F$ has finite order of growth.
The proof followed the same strategy as his proof of the second fundamental theorem and was based on estimates of 
the characteristics $A(r; F'/F)$ and $B(r; F'/F)$ of the logarithmic derivative.

\subsection{Ostrovskii's refinement} 
At the beginning of the 1960s, Ostrovskii improved Nevanlinna's result and showed that the assumption that
the meromorphic function $F$ has finite order can be replaced by the weaker condition
\begin{equation}\label{eq:Ostr}
\int_1^\infty \frac{\log^+ T(r; F)}{r^2}\, {\rm d}r<\infty,
\end{equation}
where $T(r; F)$ is the usual Nevanlinna characteristics of meromorphic functions in the complex plane.
In particular, if the function $F$ is entire, then an equivalent condition is 
\[
\int_1^\infty \frac{\log^+\log^+ M(r; F)}{r^2}\, {\rm d}r<\infty,
\]
where $M(r; F) = \max_{|z|\le r} |F(z)|$. His proof
 (see Lemma~2.2 in~\cite[Chapter~VI \S2]{GO}) was also based on a careful estimate of 
 $A(r; F'/F) +B(r; F'/F)$ \cite[Chapter III, Theorems 3.1 and 3.3]{GO} 
(the more difficult part here is the estimate of $A(r; F'/F)$, given in Lemma~3.1 in~\cite[Chapter~III \S3]{GO}). 
That is, {\em Question~2 has a positive answer for meromorphic functions satisfying condition}~\eqref{eq:Ostr}.

Apparently, Nevanlinna  hoped\footnote
{He wrote in~\cite[p.18]{Nev}:
``...dass das Restglied $R(r)$ im allgemeinen {\em von niedrigerer Grössenordnung als die Fundamentalgrösse S(r, f) ist}. 
Wir wollen bei dieser Gelegenheit auf eine allgemeine Untersuchung des Restgliedes nicht eingehen, sondern beschränken 
uns auf den einfachsten Fall, wo $f(x)$ in der {\em ganzen endlichen Ebene} eindeutig, meromorph und von {\em endlicher Ordnung} ist.''
}
to give a positive answer to Question~2 based on estimates of the logarithmic derivative
$F'/F$. This hope evaporated 
after Gol'dberg~\cite{Gold} constructed an example of a meromorphic function $F$ in $\bC$ with $S(r; F) \equiv 0$,
while $A(r; F'/F)$ grows faster than any given function tending to $\infty$ as $r\to\infty$.

\subsection{Edrei's theorem}
Another related result is due to Edrei~\cite[Theorem~1]{Edrei}. 
It states that {\em if $F$ is a meromorphic function in $\bC$ having only real $a$-, $b$-, and $c$-points 
($a, b, c$ are distinct values in the extended complex plane) and at least one of the values
$a$, $b$, and $c$ has a positive Nevanlinna deficiency, then the order of growth 
of $F$ does not exceed one}. In particular, this holds for entire functions having only real
$a$- and $b$-points, where $a$ and $b$ are distinct values in $\bC$. 

A remarkable feature of this theorem is an absence of any a priori 
assumption on the growth of $F$. Combining Edrei's theorem with the aforementioned theorem of Nevanlinna, we conclude that $F$
has bounded type in the upper and lower half-planes. Entire functions possessing this property
are called {\em entire functions of Cartwright class}.

\section{Proof of Theorem~\ref{main_thm}}

\subsection{Preliminaries}
First, we regularize the function $m$. We say that the function $m$ is {\em tame} if, in addition to being even
and non-increasing on $\bR_{\ge 0}$, it is $C^2$-smooth and satisfies the following conditions:
\begin{enumerate}
\item $m(x), |x m'(x)|, x^2|m''(x)|\to 0$ as $|x|\to \infty$;
\item it is constant on the intervals $|x|\le 1$ and $2^k-2^{k-3}\le |x| \le 2^k + 2^{k-2}$, $k\ge 3$.
\end{enumerate}

\begin{lemma}\label{lemmaA}
Suppose that $m\colon \bR \to \bR_{> 0}$ is an even continuous  function non-increasing on
$\bR_{\ge 0}$.

\smallskip\noindent {\rm (i)} If the logarithmic integral $\displaystyle \int_\bR \frac{\log^- m(x)}{1+x^2}\, {\rm d}x$ converges, 
then the function $m$ has a tame minorant $m_*\le m$ whose logarithmic integral also converges.

\smallskip\noindent {\rm (ii)} If the logarithmic integral $\displaystyle \int_\bR \frac{\log^- m(x)}{1+x^2}\, {\rm d}x$ diverges,
then the function $m$ has a tame majorant $m^*\ge m$ whose logarithmic integral also diverges.
\end{lemma}

\medskip\noindent{\em Proof of Lemma~\ref{lemmaA}}: 
First, we assume that the logarithmic integral converges.
Set
\[
m_*(x) = 
\begin{cases}
m(8), & |x|\le 1, \\
\min\{ m(2^{k+3}), 2^{-k-3}\}, & 2^k - 2^{k-3} \le |x| \le 2^k +2^{k-2}, k\ge 3.
\end{cases}
\]
This formula defines a piecewise constant function $m_*$ on certain subintervals of $\bR$.
On the remaining intervals we define the function using a rescaled version of a fixed smooth decreasing 
function on $[0, 1]$, for which the first and second derivatives vanish at the endpoints,
keeping $m_*$ even and non-increasing on $\bR_{\ge 0}$. 
Since $m_*(x)\le \frac{1}{x}$, we clearly have $m_*(x)\to 0$. Since the size of the jump between the points 
$2^k+2^{k-2}$ and $2^{k+1}-2^{k-2}$ is $o(1)$ as $k\to\infty$, after the smoothing the function $m_*$ 
will satisfy condition 1 in the definition of tameness. 
Furthermore, for $2^k \le |x| \le 2^{k+1}$, $m_*(x) \le m(2^{k+3}) \le m(x)$,
so $m_*$ is indeed a minorant of $m$. Finally, by monotonicity of the functions $m$ and $m_*$,
\begin{align*}
\int_1^\infty \frac{\log^-m(x)}{x^2}\, {\rm d}x <\infty  &\Longrightarrow \sum_k \frac{\log^- m(2^k)}{2^{k}}<\infty \\
&\Longrightarrow \sum_k \frac{\log^- m_*(2^k)}{2^{k}}<\infty  \Longrightarrow \int_1^\infty \frac{\log^- m_*(x)}{x^2}\, {\rm d}x <\infty\,,
\end{align*}
where in the second step we used that $\sum_k \frac{\log^-(2^{-k-3})}{2^k}$ clearly converges, so the extra $2^{-k-3}$ in the definition of $m_*(x)$ does not affect us.
 
The proof in the second case (the logarithmic integral diverges) is quite similar. First, 
we note that $m(x)=o(1)$ as $x\to\infty$, since the logarithmic integral is divergent and
$m$ does not increase. We set
\[
m^*(x) = 
\begin{cases}
m(0), & |x|\le 1, \\
m(2^{k-3}),  & 2^k - 2^{k-3} \le |x| \le 2^k +2^{k-2}, k\ge 3.
\end{cases}
\]
Then we smooth out the jumps of $m^*$, keeping $m^*$ even and non-increasing on $\bR_{\ge 0}$,
so that after the smoothing the function $m^*$
will satisfy condition 1 in the definition of tameness. 
Furthermore, by construction, $m^*$ is indeed a majorant of $m$, and by monotonicity of the functions $m$ and $m^*$,
the logarithmic integrals of $m$ and $m^*$ are equi-divergent:
\[
\int^\infty \frac{\log^-m(x)}{x^2}\, {\rm d}x  = \infty 
\Longrightarrow \int^\infty \frac{\log^- m^*(x)}{x^2}\, {\rm d}x  = \infty.
\]
\mbox{}\hfill $\Box$

\medskip

\begin{lemma}\label{lemmaB} Suppose that the function $m$ is tame, and 
let $w\colon \bH \to \bH( \pm m)$ be a biholomorphic map such that
$w(iy)$ is purely imaginary and  $w(iy)\to \infty$ as $y\to\infty$. 
Then the function $w'$ extends continuously to the real axis,
and $|w'(z)| \simeq 1$ everywhere in the closed half-plane $\overline\bH$. 
\end{lemma}

\medskip\noindent{\em Proof of Lemma~\ref{lemmaB}}:
We will present the proof for the domain $\bH(-m)$. For the domain $\bH(m)$, the difference is 
purely notational. 

The result follows from the classical Kellogg theorem (see, for instance, \cite[Chapter~II, Theorem~4.3]{GM}).
It yields that if $\Omega$ is a bounded Jordan domain with a $C^2$-smooth boundary, 
and $p\colon \bD\to \Omega$ is a biholomorphic map, then $p$ extends to a $C^1$-diffeomorphism 
from $\overline\bD$ to $\overline \Omega$ and $p'$ does not vanish on $\partial\bD$.  
To apply it here, we pre-compose our map $w$ with the 
Cayley transform $C(\zeta) = i (1+\zeta)/(1-\zeta)$ and post-compose it with $J(w)=1/(w + iA)$,
where $A>m(0)$, letting $p = J \circ w \circ C$. The function $p$ maps the unit disk onto 
the Jordan domain $J \bH(- m)$. 
We claim that the boundary of $J \bH(- m)$ is $C^2$-smooth. This is evident everywhere except
at the origin ($J$ maps infinity to the origin). 
To check the smoothness at the origin, we write 
\[
h(x) = \frac1{x+i(A-m(x))}\,,
\]
and let 
\[
H(s) = 
\begin{cases}
h(1/s), & s\ne 0 \\ 
0, & s=0.
\end{cases}
\]
So, we need to check that $H$ is a $C^2$-function on $[-1, 1]$. Since $H$ is continuous on $[-1, 1]$, and $C^2$
on $[-1, 0) \cup (0, 1]$, it suffices to check that the limits of $H'$ and $H''$ as $s\to \pm 0$
exist and coincide: $H'(+0)=H'(-0)$ and $H''(+0)=H''(-0)$. Let $\phi (s) = 1+is(A-m(1/s))$, $\phi(0)=1$. 
Then $H(s) = s/\phi (s)$. 
We have $ \phi'(s) = i(A-m(1/s)) + im'(1/s)/s$ and $\phi''(s) = - im''(1/s)/s^3 $. By the tameness of $m$,
we have $\phi' (s) = iA + o(1)$ and $s\phi''(s) = o(1)$, as $s\to 0$. Thus,
\begin{align*}
H'(s) &= \frac1{\phi (s)} - \frac{s\phi'(s)}{\phi(s)^2} \\
&= 1+o(1), \quad s\to 0,
\end{align*}
and 
\begin{align*}
H''(s) &= -2\,\frac{\phi'(s)}{\phi(s)^2} + 2\,\frac{s\phi'(s)^2}{\phi(s)^3} - \frac{s\phi''(s)}{\phi(s)^2} \\
&= -2iA + o(1), \quad s\to 0,
\end{align*}
proving the claim. So,  Kellogg's theorem can be applied.

Thus, $p'$ extends to a continuous non-vanishing function in the closed unit disk.
For $\zeta\to 1$ within $\bD$, we have
\[
p(\zeta) = b(\zeta-1) + o(|\zeta-1|), \quad b=p'(1)\ne 0.
\] 
Furthermore, 
\[
p'(\zeta) = J'(w(C(\zeta))) w'(C(\zeta)) C'(\zeta), \quad C'(\zeta) = 2i/(1-\zeta)^2,
\] 
and
\[
J'(w) = -1/(w+iA)^2 = - J(w)^2.
\] 
Hence, for $\zeta\to 1$, $\zeta\in \bD$, 
\begin{align*}
J'(w(C(\zeta))) C'(\zeta) &= - p(\zeta)^2 C'(\zeta) \\
&= - (b(\zeta-1)+o(|\zeta-1|))^2 
\cdot \frac{2i}{(1-\zeta)^2} \to -2ib^2,
\end{align*}
and therefore,
\begin{align*}
\lim_{z\to\infty} w'(z) &= \lim_{\zeta\to 1} w'(C(\zeta)) \\
&=\lim_{\zeta\to 1} \frac{p'(\zeta)}{J'(w(C(\zeta))) C'(\zeta)} \\
&= \frac{b}{-2ib^2} = \frac{i}{2b},
\end{align*}
proving more than what was stated in Lemma~\ref{lemmaB}. \hfill $\Box$

\subsection{Proof of Theorem~1(a)}
Proving Theorem~1(a), we assume that the function $m$ is tame. Otherwise, we replace $m$ by its tame minorant 
$m_*$ with convergent logarithmic integral $\displaystyle \int_\bR \frac{\log^- m_*(x)}{1+x^2}\, {\rm d}x$, 
which exists by Lemma~\ref{lemmaA}. 
Then $\bH(-m_*)\subset \bH(-m)$; so, if a meromorphic function 
$F$ omits three distinct values in $\bH(-m)$, a fortiori,  it omits them in $\bH(-m_*)$.

Consider the function $G=F\circ w$, where
$w\colon \bH \to \bH(-m)$ is the Riemann map, as above.
It is meromorphic in $\bH$ and omits three values therein.
Hence, by Montel's theorem, the family of meromorphic functions $\{G\circ S\colon S\in {\operatorname{Aut}}(\bH)\}$ is normal in $\bH$.
Then, following Lehto and Virtanen~\cite[Theorem 3]{LV}, we get\footnote{
For the reader's convenience, we will reproduce their nice short argument.
Consider the normal family $\mathscr M$ of functions holomorphic in $\bH$ 
and omitting the values $0$ and $1$ therein. By Mart\'y's lemma
(a version of the Arzel\'a--Ascoli theorem), the spherical derivatives $\rho_f$ of functions 
from $\mathscr M$ are locally equibounded in $\bH$. In particular,  
$\displaystyle \sup_{f\in\mathscr M} \rho_f (i) < \infty$.
The family $\mathscr M$ is invariant under the action of the group $\operatorname{Aut}(\bH) = 
{\sf SL}_2(\bR)$. So, let
\[
S(z) = \frac{\alpha z + \beta}{\gamma z + \delta}, \quad \alpha, \beta, \gamma, \delta\in\bR, \ \alpha\delta - \beta \gamma =1.
\]
Then
\[
\rho_{f\circ S} (i) = \rho_f(Si) |S'(i)| 
=\rho_f(S(i))\, \frac1{\gamma^2+\delta^2} = \rho_f(S(i))\, \operatorname{Im} S(i).
\]
Thus, $\displaystyle \sup_{f\in\mathscr M} \rho_f(z) \operatorname{Im}\, z = \sup_{f\in\mathscr M} \rho_f(i) < \infty$.
}
\begin{equation}\label{eq:LV}
\rho_G(\zeta) \lesssim \frac1{\im (\zeta)}\,, \qquad \zeta\in\bH.
\end{equation}
Whence, we get
\begin{align}\label{eq:rho_F}
\nonumber \rho_F(z) &= |(w^{-1})'(z)| \rho_G(w^{-1}(z)) \\
\nonumber & \lesssim \frac1{\im (w^{-1}(z))} 
\qquad \qquad \qquad {\rm by\ }\eqref{eq:LV}
\\
\nonumber &= \frac1{ {\rm dist\,}(w^{-1}(z), \partial \bH)} \\
&\simeq \frac1{{\rm dist\,}(z, \partial \bH(-m)) }\,, \qquad \quad z\in \bH(-m),
\end{align}
where we used Lemma \ref{lemmaB} in the second and fourth steps.
To estimate ${\rm dist\,}(z, \partial \bH(-m)) $, we use a simple geometric claim.
\begin{claim}\label{claim:geom1} Let $\phi\colon \bR\to\bR$ be a $C^1$-function with
bounded derivative $L=\| \phi'\|_\infty <\infty$. Then, for every point $(x, y)$ lying above 
the graph $\Gamma_\phi$ of $\phi$, we have
\[
{\rm dist\,}((x, y), \Gamma_\phi) \ge \frac{y-\phi(x)}{\sqrt{1+L^2}}.
\]
\end{claim}
\begin{proof} Put $\psi (t) = \phi (x) + L|t-x|$. Then $\phi (t) \le \psi (t)$.
Hence, \[ {\rm dist\,}((x, y), \Gamma_\phi) \ge {\rm dist\,}((x, y), \Gamma_\psi). \] 
Let $a=y-\phi (x)$. Then $a>0$ and ${\rm dist\,}((x, y), \Gamma_\psi)$ is the altitude 
$h$ of the right triangle with vertices $(x, \phi (x))$, $(x, y)=(x, \phi(x)+a)$, and $(x+a/L, y)$. 
This triangle has side lengths $a$, $b=a/L$, and $c=\sqrt{a^2 + (a/L)^2} = \displaystyle
a\sqrt{1+L^2}/L$. Since $hc = ab$ (both
expressions are half the area), $h=ab/c = a/\sqrt{1+L^2}$, proving the claim.
\end{proof}
Combining~\eqref{eq:rho_F} with the claim, we get
\begin{equation}\label{eq:rho_F2}
\rho_F(x+iy)  \lesssim \frac1{y+m(x)}\,, 
\qquad x+iy\in \bH(-m).
\end{equation}

Now, we are ready to see that the Nevanlinna characteristics $S(r, F)$ of $F$ in the upper half-plane
remains bounded. For this, we use its  Ahlfors--Shimizu version:
\begin{align*}
S_o(r, F) &= 
\frac1{\pi}\, \int_1^r  t\, {\rm d}t\, \int_0^\pi \left( \frac1t - \frac{t}{r^2} \right) \rho_F(te^{i\theta})^2 
\sin\theta \, {\rm d}\theta \\
&\lesssim \int_1^r {\rm d}t \int _0^\pi \frac{\sin\theta\, {\rm d}\theta}{(t\sin\theta + m(t\cos\theta))^2}
\qquad \qquad \qquad \qquad {\rm by\,} \eqref{eq:rho_F2} \\
&\lesssim \int_1^r {\rm d}t \left[ 
\int_0^{m(t)/t} \frac{\theta\, {\rm d}\theta}{m(t)^2} + \int_{m(t)/t}^{\pi/2} \frac{{\rm d}\theta}{t^2\theta} \right]
\\
&\lesssim O(1)+\int_1^r \log(t/m(t))\, \frac{{\rm d}t}{t^2}  \\
& \lesssim O(1)+\int_1^r \log\frac1{m(t)}\, \frac{{\rm d}t}{t^2} \\&= O(1),
\end{align*}
completing the proof of part (a). \hfill $\Box$

\subsection{Proof of Theorem~1(b)}\label{subsect:Thm1b}
To prove part (b), we use the standard elliptic modular function $\lambda$, which can be defined as an analytic 
continuation of the conformal map of the hyperbolic triangle with vertices $(0, 1, \infty)$ onto the upper half-plane, 
sending the vertices to $(1, \infty, 0)$. We will rely only on its most basic properties, 
which can be found in~\cite[Chapter~7, Sections~3.4, 3.5]{Ahlfors} as well as 
in~\cite[Section~23]{Akhiezer}.  
It is analytic in $\bH$, $2$-periodic, does not
attain the values $0$ and $1$, and satisfies $\lambda (\tau)\to 0$ as ${\rm Im\,}\tau\to \infty$.

\smallskip
Theorem~1(b) will be proven in several steps. 
First, using a little number theory, we show 
that, for $0<y\le 1/2$,
\[
\int_0^2 \log^+ |\la (x+iy)|\, {\rm d}x \gtrsim \log\frac1y\,.
\]
Then we fix a tame function $m$ with divergent logarithmic integral, let
$W\colon \bH(-m)\to\bH$ be the Riemann map such that $W(iy)$ is purely imaginary
and tends to $\infty$ as $y\to\infty$, and let $\Gamma = W(\bR)$. We show that 
$\Gamma$ is a graph of an even non-increasing function $n$, whose logarithmic
integral diverges. We can assume that the function
$n$ is tame (otherwise, using Lemma~\ref{lemmaA} we replace it by its
tame majorant).

We consider the function $F=\la \circ W$ that does not take values $0$ and $1$ in
$\bH(-m)$, and hence in $\bH$. We claim that $F$ is of unbounded type in $\bH$.

If not, then
\[
\int_\bR \frac{\log^+|F(t)|}{1+t^2}\, {\rm d}t < \infty\,,
\]
i.e.,
\begin{equation}\label{eq:lambda-new}
\int_\Gamma \log^+|\la (\zeta)|\, \omega_{\bH(n)}({\rm d}\zeta, \zeta_0) <\infty,
\end{equation}
where $\omega_{\bH(n)}({\rm d}\zeta, \zeta_0)$ is the harmonic measure of $\bH(n)$
evaluated at $\zeta_0 = W(i)$.

Since the function $n$ is tame, it is constant on long intervals $I_k$ centred at $2^k$ 
and of length comparable with $2^k$. Then we show that on the middle third of these
intervals  $\omega_{\bH(n)}$ is equivalent to the measure ${\rm d} t/t^2$, so the 
integral~\eqref{eq:lambda-new} is, up to a constant, at least
\[
\sum_{k\ge k_0} \int_{I_k} \log^+|\la(t+in_k)|\, \frac{{\rm d}t}{t^2}
\gtrsim \sum_{k\ge k_0} 2^{-k} \log\frac1{n(2^k)}.
\]
Finally, since the function $n$ is non-increasing, divergence of its logarithmic integral yields
divergence of the series on the right hand side, contradicting~\eqref{eq:lambda-new}.

\smallskip
Now, let us turn to the details.

\begin{lemma}\label{lemmaC}
There exists a positive numerical constant $c$ such that,
for all sufficiently small $y>0$, we have
\[
\int_0^2 \log^+ |\lambda(x+iy)|\, {\rm d}x \ge c \log (1/y).
\]
\end{lemma}

\smallskip\noindent{\em Proof of Lemma~\ref{lemmaC}}:
First, we move to the unit disk, letting $q=e^{\pi i\tau}$, $\tau = x + iy\in\bH$,  
and $\Lambda (q) = \lambda (\tau)$,
$\Lambda (0) = 0$. Then, letting $r=|q|=e^{-\pi y}$, we have\footnote
{Note that the function $\Lambda$ is not the classical universal covering
map $\bD \to \bC\setminus\{0, 1\}$ obtained from $\lambda$ via a conformal
equivalence $\bD \simeq \bH$. Rather, we use the $\lambda$-function in the variable
$q=e^{\pi i \tau}$, i.e.,  we regard $\lambda$ as a function on the quotient 
$\bH/[\tau\mapsto \tau+2] \simeq \bD\setminus\{0\}$, and then extend it to $q=0$. 
}
\[
\int_0^2 \log^+ |\lambda (x+iy)|\, {\rm d}x 
= \frac1\pi\, \int_{-\pi}^\pi \log^+|\Lambda(re^{i\theta})|\, 
{\rm d}\theta.
\]
Hence, for any $a\in \bC\setminus\{0, 1\}$,
\begin{align*}
\int_0^2 \log^+ |\lambda (x+iy)|\, {\rm d}x  &\ge
\frac1\pi\, \int_{-\pi}^\pi \log^+ |\Lambda(re^{i\theta}) - a|\, 
{\rm d}\theta - C(a) \\
&\ge 2\sum_{|q_k|\le r} \log\frac{r}{|q_k|} - C_1(a) 
\qquad \qquad ({\rm by\ the\ Jensen\ formula}),
\end{align*}
where $\{q_k\}$ denotes
the set of $a$-points of $\Lambda$ (counted with multiplicities), so, $|q_k|=e^{-\pi {\rm Im\,}\tau_k}$, where 
$\{\tau_k\}$ is the set of $a$-points $\lambda$ in $\{-1\le \re \tau <1\}$. Then
$\log (r/|q_k|) = \pi (\im \tau_k - y)$, and it will suffice to show that
\begin{equation}\label{eq:*}
\sum_{\im \tau_k \ge 2y} \im \tau_k  \gtrsim \log\frac1y, \qquad y\downarrow 0.
\end{equation}
It remains to make the necessary counting.

The group ${\sf SL}_2(\bZ)$ acts on the upper half-plane $\bH$,
\[
(M, w) \mapsto Mw = \frac{\alpha w + \beta}{\gamma w + \delta}, 
\quad M = \left( \begin{matrix}
\alpha & \beta \\
\gamma & \delta
\end{matrix} \right), \ w\in\bH.
\]
The modular function $\lambda$ is invariant under the subgroup $\Gamma (2)$ of ${\sf SL}_2(\bZ)$,
which consists of integer-valued matrices $M$ satisfying $\alpha \delta - \beta\gamma =1$, and
\[
\left( \begin{matrix}
\alpha & \beta \\
\gamma & \delta
\end{matrix} \right)
\equiv
\left(
\begin{matrix}
1 & 0 \\
0 & 1
\end{matrix}
\right)
\quad ({\rm mod\,} 2).
\]

Recall that $\Gamma (2)$ is a subgroup of index $6$ in 
${\sf SL}_2(\bZ)$, and that, for $\phi\in {\sf SL}_2(\bZ)$, $\lambda \circ \phi$ can take 
only one of the following six values
\[
\lambda,  \quad \frac1{\lambda}, \quad 1-\lambda, \quad  
\frac1{1-\lambda}, \quad \frac{\lambda}{1-\lambda}, \quad \frac{\lambda-1}\lambda.
\]
Fix $w=i$ and consider its orbit $G(w)=\{\phi w\colon \phi\in {\sf SL}_2(\bZ)\}$. For any $z\in G(w)$ we have that $\lambda(z)$ is equal to one of six numbers \[
a,  \quad \frac1a, \quad 1-a, \quad  
\frac1{1-a}, \quad \frac{a}{1-a}, \quad \frac{a-1}a,
\] 
where $a = \lambda(w)$. If we are able to show that
$$\sum_{\substack{z\in G(w):\\ |{\rm Re}\,z| < 1, {\rm Im}z > 2y}}{\rm Im}\,z \gtrsim \log \frac{1}{y}$$
then by the pigeonhole principle the same would hold for at least one of six values above, thus giving us the desired claim (note that, potentially, which exact value we need will depend on $y$ but as long as the total set of possible values is finite and fixed this is not a problem).

For a matrix $M = \left( \begin{matrix}
\alpha & \beta \\
\gamma & \delta
\end{matrix} \right)\in {\sf SL}_2(\bZ)$ and $w=i$, we have 
\begin{equation}\label{Mw}
Mw =  \frac{\alpha i + \beta}{\gamma i + \delta} 
= \frac{\beta \delta + \alpha \gamma}{\gamma^2 + \delta^2} + \frac{i}{\gamma^2 + \delta^2}\,.
\end{equation}

It might happen that $M_1w = M_2w$ for distinct matrices $M_1$ and $M_2$. However, in this case $M_2^{-1}M_1 w =w$ and this cannot happen too often. Specifically, $Mw = w$ has exactly four solutions
$$M = \left( \begin{matrix}
1 & 0 \\
0 & 1
\end{matrix} \right), \left( \begin{matrix}
-1 & 0 \\
0 & -1
\end{matrix} \right), \left( \begin{matrix}
0 & -1 \\
1 & 0
\end{matrix} \right), \left( \begin{matrix}
0 & 1 \\
-1 & 0
\end{matrix} \right),$$
which can be easily seen from \eqref{Mw}: we must have $\gamma^2+\delta^2 = 1$, thus $(\gamma, \delta) = (\pm 1, 0)$ or $(\gamma, \delta) = (0, \pm 1)$ and in all four cases we can uniquely determine $\alpha, \beta$ from $\alpha\delta - \beta\gamma = 1$ and $\alpha\gamma + \beta\delta = 0$.

Thus, if, instead of summing over the orbit $G(w)$, we sum over all matrices in $ {\sf SL}_2(\bZ)$ then we will count each point exactly four times, which does not affect the final conclusion of it being $\gtrsim \log \frac{1}{y}$. So, it remains to show that
$$\sum \frac{1}{\gamma^2+\delta^2} \gtrsim \log \frac{1}{y},$$
where the sum is taken over all integers $\alpha, \beta, \gamma, \delta$ such that
\[
\alpha\delta-\beta\gamma = 1, \quad |\beta\delta + \alpha\gamma| < \gamma^2+\delta^2, \quad 
\gamma^2+\delta^2\le \frac{1}{2y}.
\]

Let us fix coprime numbers $\gamma, \delta$. By Euclid’s algorithm, there exists a pair $\alpha, \beta$ such that $\alpha \delta - \beta \gamma =1$. All other pairs are obtained by adding or subtracting $\gamma$ to $\alpha$ and $\delta$ to $\beta$, in which
case $\beta\delta+\alpha\gamma$ will change by $\gamma^2+\delta^2$ (this corresponds to shifting $Mw$ by $1$ to the left or to the right). Thus, for each such pair $\gamma, \delta$, we can find at least one pair $(\alpha, \beta)$ such that $|\beta\delta + \alpha\gamma| < \gamma^2+\delta^2$. Hence, it suffices to show that 
$$\sum_{\substack { \gamma^2 + \delta^2 \le 1/(2y), \\ (\gamma, \delta)=1}}\,
\frac1{\gamma^2+\delta^2}\gtrsim \log \frac{1}{y}.$$

Finally, for a large number $X$, the number of coprime pairs 
$(\gamma, \delta)$ such that $X \le \gamma^2 + \delta^2 \le 2X$ is of order $X$ (because 
the number of pairs which are coprime has a positive density, see, for instance, \cite[Theorem~ 332]{HW}). 
Splitting the sum we are estimating into
dyadic blocks, we see that the sum is of order $\log (1/y)$, as required. \hfill $\Box$

\medskip 
Proving Theorem~1(b), we assume that the function $m$ is tame. Otherwise, we replace $m$ by its tame majorant 
$m^*$ with divergent logarithmic integral, which exists by Lemma~\ref{lemmaA}. 
Then $\bH(-m^*)\supseteq \bH(-m)$; so, if a meromorphic 
function $F$ omits  three distinct values in $\bH(-m^*)$, {\em a fortiori} it omits the same values in $\bH(-m)$.

As in the proof of Theorem~1(a), let $w\colon \bH \to \bH(-m)$ be the Riemann map such that
$w(iy)$ is purely imaginary and  $w(iy)\to \infty$ as $y\to\infty$. Set $W=w^{-1}$ and
$\Gamma = W (\bR)$, that is, $\im w(z) = 0$ for $z\in \Gamma$. Clearly, $\Gamma$ is a smooth curve, symmetric
with respect to the imaginary axis. Recall that, due to tameness of $m$, by Lemma~\ref{lemmaB},
the mapping $W$ is continuously differentiable up to the boundary.

\begin{claim}\label{lemmaD}
Let $W=U+iV$. Then, the function $x\mapsto U(x, 0)$ is strictly increasing on $\bR$, while the function
$x\mapsto V(x, 0)$ is non-increasing on $\bR_{\ge 0}$.
\end{claim}
\begin{proof}
To prove the first statement, consider the function $h_s(z) = V(z+is)-V(z)$, which is
harmonic in $\bH(-m)$. On $\partial \bH(-m)$, $h_s >0$. Moreover, $h_s$ is bounded because,
by Lemma~\ref{lemmaB}, $W'$ is bounded. Thus, by the Phragm\'en--Lindel\"of principle, 
$h_s\ge 0$ in $\bH(-m)$, that is, $V(z+is) \ge V(z)$ therein. Letting $s\to 0$, we get 
$V_y \ge 0$. Since $V_y$ is harmonic and non-constant, it cannot have a minimum inside $\bH(-m)$,
i.e., $V_y>0$ in $\bH(-m)$. By Cauchy--Riemann, $U_x = V_y>0$, proving the first statement.

The proof of the second part is quite similar.
Consider the partial derivative $V_x(x, y)$. It is a bounded harmonic function 
in $\bH (-m)$ (boundedness follows from Lemma~\ref{lemmaB}).
By the symmetry of the map $W$, we have $V(-x, y)=V(x, y)$, and therefore the function $V_x$
vanishes on the imaginary axis. The function $V$ is positive in $\bH(-m)$ and vanishes 
on $\partial \bH(-m)$. Then its gradient $\nabla V$ points into $\bH(-m)$. Since $m$ is non-increasing,
this implies that $V_x (x, - m(x)) \le 0$ for $x\ge 0$ (and similarly, $\ge 0$ for $x\le 0$).  Thus, by the 
Phragm\'en--Lindel\"of principle, $V_x (x, y) \le  0$ in $\bH(-m) \bigcap \bigl\{ {\rm Re\,}z>0 \bigr\}$. 
\end{proof}

By Claim~\ref{lemmaD}, the curve $\Gamma$ is a graph of an even smooth function $n\colon \bR\to\bR_{>0}$, 
non-increasing on  $\bR_{\ge 0}$: $\Gamma = \{x+i n(x)\colon x\in \bR\}$.

\begin{claim}\label{claim:n}
For $x\in \bR$, 
\begin{equation}\label{eq:n} 
n(x) \le Cm(x/C)
\end{equation}
with $C= \| W' \|_\infty =\sup\{| W'(\zeta)|\colon \zeta\in\overline\bH(-m)\}$.
\end{claim}
\begin{proof}
Since both functions $m$ and $n$ are even, it suffices to check \eqref{eq:n} 
only for $x\ge 0$. Let, as above, $W=U+iV$. Then $V(t)=n(U(t))$, $t\in\bR$. Since
\[
V(t)={\rm Im\,} W(t) = {\rm Im\,}[ W(t)-W(t-im(t) ],
\]
we have
\[
V(t) \le | W(t) - W(t-im(t))| \le C m(t).
\]
Similarly, 
\[
U(t) = {\rm Re\,} W(t) = {\rm Re\,}[W(t)-W(0)],
\] 
whence, for $t\ge 0$, $U(t)\le |W(t)-W(0)|\le Ct$, 
and therefore, $m(t) \le m(U(t)/C)$. Combining both estimates, we obtain
\[
n(U(t)) = V(t) \le C m(t) \le Cm(U(t)/C).
\] 
Letting $x=U(t)$, we get the claim.
\end{proof}

Now we are ready to prove Theorem~1(b). Assume that 
\[ \int_1^\infty \frac{\log^- m(t)}{t^2}\, {\rm d}t = +\infty, \] 
that $m$ is tame, and let $F=\lambda \circ W$. The function $F$ is analytic in $\bH(-m)$ and does not take the values $0$ and $1$
there. We aim to show that $F$ is of unbounded type in $\bH$. Suppose that 
it is not true, and $F$ is of bounded type in $\bH$. Then, 
by the classical Nevanlinna representation~\cite[Chapter~VI, Theorem~3.1]{GO},
\begin{equation}\label{eq:logintegral}
\int_{\bR} \frac{\log^+|F(t)|}{t^2+1}\, {\rm d}t <\infty,
\end{equation}
and 
\[
\log^+|F(\zeta)| 
\le  \frac{\im \zeta}\pi \int_\bR \frac{\log^+ |F(t)|}{|t-\zeta|^2} \, {\rm d}t + c\, {\rm Im\,}\zeta,
\]
with
\[
c = \lim_{y\to +\infty} \frac1y\, \log^+ |F(iy)|.
\]
Moreover, since, for a given $x$, $\lambda (x+iy)$ tends to $0$ as $y\to\infty$,
$F(iy) = (\la\circ W)(iy)\to 0$, and $c=0$. So, 
$\log^+ |F|$ is majorized by its Poisson integral  
\begin{align}\label{eq:star}
\nonumber \log^+|F(\zeta)| 
&\le  \frac{\im \zeta}\pi \int_\bR \frac{\log^+ |F(t)|}{|t-\zeta|^2} \, {\rm d}t \\
&=  \int_\bR \log^+ |F(t)|\, \omega_\bH ({\rm d}t, \zeta) < \infty, \qquad \zeta\in\bH.
\end{align}
Here and elsewhere, $\omega_{\bG}({\rm d}z, \zeta)$ denotes the harmonic measure of the domain
$\bG$ evaluated at $\zeta\in\bG$.
Letting $\zeta=w(z)$ in~\eqref{eq:star}, we get
\begin{equation}\label{eq:starstar}
\log^+ |\lambda (z)| \le
\int_\Gamma \log^+ |\lambda (z')|\, \omega_{\bH (n)}({\rm d}z', z) < \infty, 
\qquad z\in \bH(n),
\end{equation}
where, as above,  $\Gamma =\{x+i n(x)\colon x\in \bR\}=\partial \bH(n)$.
By~Claim~\ref{claim:n},
\[ \int_1^\infty \frac{\log^- n(t)}{t^2} \, {\rm d}t = +\infty. \] 

Let $n^*$ be a tame majorant of $n$ with divergent logarithmic integral which exists by Lemma~\ref{lemmaA}.
Then $\bH(n^*) \subset \bH(n)$. Let $\Gamma^* = \partial \bH(n^*) = \{x+in^*(x)\colon x\in\bR \}$,
and let $w^*\colon \bH \to \bH(n^*)$ be the Riemann map with our usual convention ($w^*$ is purely imaginary,
and $w^*(i\eta)\to\infty$ as $\eta\to\infty$).
\begin{claim}\label{claim:PL}
For $z\in \bH(n^*)$, 
\begin{equation}\label{eq:3stars}
\int_{\Gamma^*} \log^+ |\lambda (z')|\, \omega_{\bH (n^*)}({\rm d}z', z)
\le \int_\Gamma \log^+ |\lambda (z')|\, \omega_{\bH (n)}({\rm d}z', z) < \infty.
\end{equation}
\end{claim}
\begin{proof} Denote by $h(z)$ the right hand side of~\eqref{eq:3stars}. 
This is a non-negative harmonic function in $\bH(n)$
that majorizes $\log^+|\lambda|$ in $\overline\bH(n)$. Returning to the upper half-plane, we let 
$u=\log^+|\lambda \circ w^*|$ and $v=h\circ w^*$. The non-negative functions $u$ and $v$ are correspondingly 
subharmonic and harmonic in $\bH$, continuous in $\overline\bH$, and $u\le v$ in $\overline\bH$. 
Furthermore, by the Herglotz -- F. Riesz representation theorem applied to the positive harmonic function $v$, we have
\[
v(\zeta) = c\eta + \frac{\eta}{\pi}\, \int_{\bR} \frac{v(t)}{|t-\zeta|^2}\, {\rm d}t,
\qquad \zeta = \xi + i\eta,
\]
with $c\ge 0$. Therefore,
\[
\frac{\eta}{\pi}\, \int_{\bR} \frac{u(t)}{|t-\zeta|^2}\, {\rm d}t
\le \frac{\eta}{\pi}\, \int_{\bR} \frac{v(t)}{|t-\zeta|^2}\, {\rm d}t \le v(\zeta)\,,
\qquad \zeta\in\bH,
\]
which is nothing but~\eqref{eq:3stars} with $z=w^*(\zeta)$.
\end{proof}

To arrive at a contradiction,  it remains to show that the integral on the left hand side 
of~\eqref{eq:3stars} is infinite. 
Fix $z\in \bH(n^*)$. The function $n^*$ takes the constant value 
$n_k=n^*(2^k)$
on the interval $I_k=[2^k - 2^{k-3}, 2^k + 2^{k-2}]$, $k\ge 3$. 
Hence, the curve $\Gamma^*=\{t+in^*(t)\}$ 
has long horizontal intervals $\Gamma_k$ over $I_k$ of height $n_k$. 
Next, we claim that on the middle third of $\Gamma_k$ the harmonic measure  
$\omega_{\bH (n^*)}({\rm d}z', z)$ is uniformly in $k$ comparable with the measure ${\rm d t}/t^2$.
\begin{claim}\label{claim:harm_measure}
Let $J_k$ be the middle third of $I_k$. Then
there exists $k_0=k_0(z)$ and constants
$0<c(z)< C(z)<\infty$ {\rm (}depending on $z$ but not on $k${\rm )} such that, for every $k\ge k_0$ and every Borel set
$A\subset J_k$, 
\[
c(z) \int_A \frac{{\rm d}t}{t^2} \le \omega_{\bH (n^*)}(A+in_k, z) 
\le C(z) \int_A \frac{{\rm d}t}{t^2}
\]
\end{claim}
\begin{proof}Indeed, let, as above, $w^*=u^*+iv^*\colon \bH \to \bH(n^*)$ be the Riemann map with our usual normalization:
$u^*(i\eta) = 0$ and $v^*(i\eta)\to \infty$ as $\eta\to +\infty$. 
The function $w^*$ maps an interval $T_k\subset \bR$ onto the boundary segment $S_k=\{x+in_k\colon x\in J_k\}$.
Then $u^*(T_k)=J_k$, and, by Lemma~\ref{lemmaB}, $|T_k| \simeq |J_k|$ and $|u^*(t)| \simeq |t|$ for $t\in T_k$. 
By the conformal invariance of the harmonic measure,
\[ \omega_{\bH (n^*)}(E, z) = \omega_{\bH} ((w^*)^{-1}(E), \zeta), \qquad z=w^*(\zeta).\] 
Hence, for $A\subset J_k$ and $E=\{t+in_k\colon t\in A\}$,
\[
\omega_{\bH (n^*)}(E, z) = \frac1\pi\,
\int_{(w^*)^{-1}(E)} \frac{\eta}{(t-\xi)^2 + \eta^2}\, {\rm d}t, \quad \xi+i\eta = \zeta.
\] 
Furthermore, $|(w^*)^{-1}(E)| \simeq |A|$, $|t|\simeq 2^k$, and for $k$ large enough,
$ \eta/((t-\xi)^2 + \eta^2) \simeq 1/t^2$ (with implicit constants depending on $\zeta$, that is, on $z$),
proving the claim.
\end{proof}

\smallskip Finally, we return to the integral on the left hand side of~\eqref{eq:3stars}. By Claim~\ref{claim:harm_measure}, 
it is 
\[
\gtrsim  \sum_{k\ge k_0} \int_{J_k} \log^+ |\lambda (t+in_k)|\,
\frac{{\rm d}t}{t^2}
\]
The $k$-th term of this sum is
\[
\gtrsim 2^{-2k}\, \int_{J_k} \log^+ |\lambda (t+in_k)|\, {\rm d}t,
\]
which, by Lemma~\ref{lemmaC} and $2$-periodicity of $\lambda$, is 
\[
\gtrsim 2^{-2k} \cdot 2^k \log\frac1{n_k} =
 2^{-k} \log \frac1{n^*(2^k)}.
\]
Since the function $n^*$ is non-increasing, divergence of the series
\[
\sum_k  2^{-k} \log \frac1{n^*(2^k)} = \infty
\]
follows from the divergence of the integral
\[
\int_1^\infty \frac{\log^- n^*(t)}{t^2}\, {\rm d}t  = \infty.
\]
This proves the divergence of the integral on the left hand side of~\eqref{eq:3stars}, and completes the
proof of the theorem. \mbox{}\hfill $\Box$

\section{Example of Yixin He and Teng Zhang: a brief discussion}
\label{sect:brief}

We briefly sketch their construction. The function $F$ has the form
$\lambda\circ\phi$, where $\lambda$
is the modular function defined in Section~\ref{subsect:Thm1b}, and $\phi$ is a conformal map
$\phi\colon \bH\to\Omega\subset\bH$. To construct $\Omega$, consider the set
\[
S =\{ z\in \bH\colon  \rm{Im\,}\lambda(z)=0\}.
\]
This set consists of disjoint open semi-circles and vertical lines $\bH$; 
some of  their closures are tangent to each other at rational points of $\bR$. The boundary 
$\gamma=\partial\Omega$ is a curve
consisting of some of these open semi-circles and points on the real line which do not accumulate. 
It is not difficult to see that the function $\la\circ\phi$, as a map to the Riemann sphere, 
has continuous real boundary values on $\bR$ and extends to a meromorphic function $F$ 
by Schwarz reflection. All $0$-, $1$-, and $\infty$-points of the function $F$ are real.

Moreover, we can make the region $\Omega$ asymptotically very close to $\bH$. 
Specifically,
one can choose $\Omega$ so that $\partial \Omega$ lies in an arbitrarily
narrow corridor of the form \[ \{ x+iy\colon 0\le y<m(x), x\in\bR\}, \]
where $m(x)$ is
a positive continuous even function, non-increasing on $\bR_{\ge 0}$, with
divergent logarithmic integral.
By the argument in the proof of Theorem~\ref{main_thm}(b),  this choice 
implies that $F$ is not of bounded type in $\bH$.

Lastly, let us mention that the  version of Question~\ref{quest_Nevanlinna}
for entire functions remains open.

\bigskip
\medskip

\noindent Alexandre Eremenko,
Mathematics Department, Purdue University, West Lafayette, IN 47907 USA, 
{\tt eremenko@purdue.edu}

\medskip\noindent Aleksei Kulikov,
University of Copenhagen, Department of Mathematical Sciences, Universitetsparken 5, 2100 Copenhagen, Denmark,
{\tt  ak@math.ku.dk}, {\tt lyosha.kulikov@mail.ru}

\medskip\noindent Mikhail Sodin,
School of Mathematics, Tel Aviv University, Tel Aviv 69978, Israel,
\newline {\tt sodin@tauex.tau.ac.il}

\end{document}